\documentclass[preprint,10pt]{elsarticle}

\usepackage{bm}
\usepackage{amsfonts}
\usepackage{amscd}
\usepackage{latexsym}
\usepackage{pifont}
\usepackage{fleqn}
\usepackage{pifont}
\usepackage{natbib}
\usepackage{subfigure}
\usepackage[margin=5pt,font=small,labelfont=bf]{caption}
\usepackage{epsf}
\usepackage{indentfirst}
\usepackage{mathrsfs}
\usepackage{amsmath}
\usepackage{color,xcolor}
\usepackage{tikz}
\usepackage{float}
\usepackage{geometry}
\usepackage{hyperref}
\usepackage{float}
\usepackage{booktabs}
\usepackage{amssymb}
\usepackage{multirow}
\usepackage{dsfont}
\usepackage{appendix}
\usepackage{soul}
\usepackage{amsmath}
\DeclareMathAlphabet{\mathcal}{OMS}{cmsy}{m}{n}
\DeclareSymbolFont{largesymbols}{OMX}{cmex}{m}{n}
\usepackage{graphicx}
\usepackage{epstopdf} 
\usepackage{mathptmx}      
\biboptions{numbers,sort&compress}
\usepackage{chngcntr}
\usepackage{graphicx,psfrag}                    

\renewcommand{\thefootnote}{\fnsymbol{footnote}}
\newtheorem{theorem}{Theorem}[section]

\newtheorem{remark}[theorem]{Remark}

\newcommand{\ep}{\hfill\rule{0.15cm}{0.35cm}\vskip 0.3cm}
\newenvironment{proof}[1][Proof.]{ \begin{trivlist}
\item[\hskip \labelsep {\bfseries #1}]}{\ep\end{trivlist}}

\numberwithin{equation}{section}

\begin{document}

\begin{frontmatter}
\title{Local convergence analysis of L1/finite element scheme for a constant delay reaction-subdiffusion equation with uniform time mesh}
\renewcommand{\thefootnote}{\fnsymbol{footnote}}
\renewcommand{\thefootnote}{\fnsymbol{footnote}}
\author[xtu]{Weiping Bu\corref{cor}}\ead{weipingbu@xtu.edu.cn}
\author[xtu]{Xin Zheng}
\address[xtu]{School of Mathematics and Computational Science$\And$Hunan Key Laboratory for Computation and Simulation in Science and Engineering,
Xiangtan University, Hunan 411105, China}
\cortext[cor]{Corresponding author.}
\begin{abstract}
The aim of this paper is to develop a refined error estimate of L1/finite element scheme for a reaction-subdiffusion equation with constant delay $\tau$ and uniform time mesh. Under the non-uniform multi-singularity assumption of exact solution in time, the local truncation errors of the L1 scheme with uniform mesh is investigated. Then we introduce a fully discrete finite element scheme of the considered problem. Next, a novel discrete fractional Gr\"onwall inequality with constant delay term is proposed, which does not include the increasing Mittag-Leffler function comparing with some popular other cases. By applying this Gr\"onwall inequality, we obtain the pointwise-in-time and piecewise-in-time error estimates of the finite element scheme without the Mittag-Leffler function. In particular, the latter shows that, for the considered interval $((i-1)\tau,i\tau]$, although the convergence in time is low for $i=1$, it will be improved as the increasing $i$, which is consistent with the factual assumption that the smoothness of the solution will be improved as the increasing $i$. Finally, we present some numerical tests to verify the developed theory.


\end{abstract}
\begin{keyword}
Reaction-subdiffusion equation with constant time delay, L1 scheme, Finite element method, Piecewise-in-time error estimate

\end{keyword}
\end{frontmatter}
\section{Introduction}
\label{Sec1}
As is known to all, fractional operator has nonlocal and memory properties, which allows fractional differential equations to more accurately describe the complex diffusion behavior of systems in simulating abnormal diffusion phenomena. In the last several decades, fractional models have been extensively applied in science and engineering \cite{Hilfer2000Applications,Magin2010Fractional,Jajarmi2018A,Sun2018A,Chatterjee2021A}. It is worth noting that there are usually some models established in real life including biology, epidemic and predation \cite{Carvalho2017A, Rihan2018Applications,Ma2020Spatiotemporal, Kumar2023The} which depend on the history. In recent years, some models with both fractional operators and time delay have attracted the interest of a large number of  researchers\cite{Ding2018Analytical,Alquran2019Delay,Rihan2019A,Alidousti2019Stability,Barman2022Modelling,Yao2023Stability,Rihan2023Fractional,Mvogo2023Dynamics}. Compared with the theoretical studies, the numerical methods of these models are still quite limited. In \cite{Tan2022L1}, an L1/finite element
method with the symmetrical time graded mesh was proposed to solve the following constant time delay subdiffusion equation
\begin{equation}\label{eqs1 1}
      _0^CD_t^\alpha u(x,t)=p\partial_x^2u(x,t)+a(x)u(x,t)+bu(x,t-\tau)+f(x,t),\quad (x,t)\in \Omega \times (0, K\tau],
\end{equation}
with the initial and boundary conditions
\begin{equation}\label{eqs1 2}
    u(x,t)=\phi(x,t),\quad (x,t)\in \bar{\Omega} \times [-\tau, 0] \quad \text{and} \quad u(x,t)|_{\partial \Omega}=0,\quad t \in [-\tau, K\tau],
\end{equation}
where $\Omega = (0, L)$, $L, p$ are some positive constants, $a(x)\leq0, f(x,t)$ are continuous functions, $\tau>0$ is the time delay parameter, $K$ is a given positive integer, $b$ is a constant, $\phi(x,t) $ is a continuous function in $\bar{\Omega} \times \left[-\tau,0\right]$, and the Caputo fractional derivative of order $0< \alpha <1$ is defined by

\begin{equation}\label{eqs1 4}
    _0^CD_t^\alpha u(x,t)=\int_0^t\omega_{1-\alpha}(t-s)\partial_su(x,s)ds,\quad\omega_\alpha(t)=\frac{t^{\alpha-1}}{\Gamma(\alpha)}.
\end{equation}
In this paper, we reconsider the L1/finite element method for problem (\ref{eqs1 1})--(\ref{eqs1 2}) based on the uniform time mesh and establish the local error estimate.

At present, the theoretical study of time delay fractional diffusion equations has made some progress. In \cite{Ouyang2011Existence}, Ouyang discussed some sufficient conditions of existence of solutions for a nonlinear fractional difference equation with time delay.
By using some other technique, Zhu et al. \cite{Zhu2019Existence} declared that they improved and generalized the above results mentioned in \cite{Ouyang2011Existence}. Prakash et al. \cite{Prakash2020Exact} devised a method to search the exact solution of the constant delay time fractional diffusion equations. Karel et al. \cite{Van2022On} discussed the existence and uniqueness of the exact solutions for a variable-order time fractional diffusion equation with delay. To numerically solve the time fractional diffusion equation with delay, there have been some methods such as the finite difference method, the spectral element method and the finite element method \cite{Zhang2017Analysis,Li2018Convergence,Hendy2020Convergence,Pimenov2017A,Dehghan2018A,Zaky2023An,Peng2023Convergence,Peng2024Unconditionally}.
In particular, in \cite{Pimenov2017On,Abbaszadeh2021A}, the authors have extended the discussion of numerical solution to constant delay subdiffusion equation with distribution order time fractional derivative. Furthermore, there was also some researchers who proposed fast algorithm for the subdiffusion equation with delay to reduce the computation cost \cite{Zhao2018A}.

In recent, Tan et al. \cite{Tan2022L1} investigate the problem (\ref{eqs1 1})--(\ref{eqs1 2}) and discover that the solution of this problem exhibits multiple singularities.
Subsequently, for a time fractional ordinary differential equation with constant time delay, Cen and Vong \cite{Cen2023The} obtain a sharp regularity of the solution. The result shows that  non-uniform multi-singularity will appear near the integer multiples of delay. To overcome this singularity, two effective corrected $L$-type methods are developed in \cite{Cen2023Corrected}.
For a generalized time delay fractional diffusion equation, Bu et al. \cite{Bu2024Finite} use the method of variable separation and the Laplace transform method to derive the analytical solution, and obtain the temporal regularity of the solution similar to \cite{Cen2023The}. Furthermore, based on the symmetric time graded mesh,  an effective L1/finite element method is established for the considered problem and an optimal convergence rate is obtained in time.

It is worth pointing out that, in order to achieve optimal convergence result,  the symmetric time graded mesh parameter $r\geq \frac{2-\alpha}{\alpha}$ needs to be ensured in \cite{Tan2022L1}. However, when $\alpha$ goes to small, the mesh will become dense near each delay point. In this case, some round-off errors may occur leading to additional difficulties in practical applications.
In fact, Refs. \cite{Cen2023The,Bu2024Finite} discover that some time fractional differential equations with constant delay time have the following regularity result
\begin{equation}\label{eq240530a}
  |\partial_{t}u_{k\tau}(t)|\leq C\bigg(1+\big(t-(k-1)\tau\big)^{k\alpha-1}\bigg),
\end{equation}
where $t\in \left((k-1)\tau, k\tau \right],k=1,2,\cdots,K$ and $\partial_{t}$ denotes the first temporal derivative. This implies that the smoothness of the solution to problem (\ref{eqs1 1})--(\ref{eqs1 2}) will be improved from $\left((k-1)\tau, k\tau \right]$ to $\left(k\tau, (k+1)\tau \right]$. Therefore, although Refs. \cite{Tan2022L1,Bu2024Finite} only obtain the global temporal convergence of order $\alpha$ under the uniform mesh, we still believe that the convergence order can be improved. Besides, we note that there have been some convergence studies on L1 scheme of fractional differential equations based on the singularity exact solution and the uniform mesh \cite{Li2021Sharp,Jiang2024Local}, meanwhile little work discusses the case with delay. Thus, the above reasons motivate us to consider local temporal convergence analysis on L1 scheme of constant delay subdiffusion equation (\ref{eqs1 1})--(\ref{eqs1 2}) under the multi-singularity  (\ref{eq240530a}) and the uniform mesh.

In this paper, in order to improve the convergence result of Ref. \cite{Tan2022L1} and  establish local convergence analysis of the L1/finite element scheme for the problem (\ref{eqs1 1})--(\ref{eqs1 2}) under the multi-singularity and the uniform time mesh, we propose a novel fractional Gr\"onwall inequality firstly. Actually, this inequality can be regarded a generalization from the case of classic time fractional diffusion problem to that of subdiffusion equation with constant time delay. Unlike the Gr\"onwall inequality developed in \cite{Tan2022L1}, the present inequality exhibits consistent estimate, i.e., this inequality does not contain increasing functions such as Mittag-Leffler function. Then, using the developed fractional Gr\"onwall inequality, we obtain pointwise-in-time error estimate of the considered numerical scheme. The convergence result implies that
$$
\max_{(i-1)N+1\leq n\leq iN} \|u_h^n - U_h^n\|_0 \lesssim \rho^{\min\{1,i\alpha\}}+h^2,\quad i=1,2,\cdots,K,
$$
where $\rho,h$ are the temporal and spacial sizes, $u_h^n, U_h^n$ denote theoretical and numerical solutions, respectively. It indicates that, for $i\alpha \leq 1$, the local  temporal convergence order is $i\alpha$ in the sense of maximum time error on the interval $((i-1)\tau, i\tau]$ rather than $\alpha$ given in \cite{Tan2022L1}. It is obvious that this result is more reasonable, because the exact solution will becomes smoother in time when $i$ increase.

This article are organized as follows. In Section2, We introduce some related symbols, regularity assumptions and discrete schemes for problem (\ref{eqs1 1})--(\ref{eqs1 2}), and derive the local truncation error of L1 scheme on uniform mesh. In Section3, we deduced a novel discrete fractional Gr\"onwall inequality in detail. In Section4, We using Gr\"onwall inequality to study the convergence. In Section5, a numerical experiment are displayed to check the theoretical results.

\textbf{Notation}: In this paper, $C,C_1,C_2$ is a general constant, which can take different values in different places.

\section{Numerical scheme of (\ref{eqs1 1})--(\ref{eqs1 2})}
\label{Sec2}
In this section, we introduce the numerical scheme of (\ref{eqs1 1})--(\ref{eqs1 2}). For given positive integer $N$ and interval $[-\tau, K\tau]$, let $t_n=\frac{n\tau}{N},\ n=-N,-N+1,\cdots,KN$ to divide this interval into $(K+1)N$ subintervals. Denote $\rho=\frac{\tau}{N}$, $u^{n}=u(t_{n})$ and  $\begin{aligned}\nabla_tu^n=u(t_n)-u(t_{n-1})\end{aligned}$. Then, using the famous L1 formula, the Caputo fractional derivative $_0^CD_t^\alpha u(t),0<\alpha<1$ can be approximated as \cite{Tan2022L1,Bu2024Finite}
\begin{equation}\label{eqs3 1}
    \begin{aligned}
_0^CD_t^\alpha u(t_n)
&=\sum_{k=0}^{n-1}a_{k}\nabla_tu^{n-k}+(R_t^\alpha)^n\\
&:=D^\alpha_N u^n + (R_t^\alpha)^n, \   1\le n\le KN,
\end{aligned}
\end{equation}
where the coefficients $a_{k}=\frac{\left(k+1\right)^{1-\alpha}-k^{1-\alpha}}{\Gamma(2-\alpha)\rho^\alpha},k=0,1,\cdots$ are positive and monotone decreasing, and $(R_t^\alpha)^n$ is the corresponding local truncation error.

Without loss of generality, suppose that the regularity of the solution to (\ref{eqs1 1})--(\ref{eqs1 2}) in the remainder of this paper satisfies
\begin{equation}\label{eqs2 1}
         |\partial_{x}^{m}u_{k\tau}(x,t)|\leq C_1, \
         |\partial_{t}^{m}u_{k\tau}(x,t)|\leq C_2\bigg(1+\big(t-(k-1)\tau\big)^{k\alpha-m}\bigg),\  m=0,1,2,
\end{equation}
where $(x,t)\in \bar{\Omega} \times \left((k-1)\tau,k\tau\right],k=1,2,\cdots,K$. For more details on the discussion of regularity, we refer the readers to \cite{Tan2022L1,Cen2023The,Bu2024Finite}.
Now we give out the local truncation error $(R_t^\alpha)^n$ of L1 scheme on uniform mesh. 

\newtheorem{lem}{Lemma}[section]
\begin{lem}\label{jullemma5.2} Assume that $u$ be the exact solution of (\ref{eqs1 1})-(\ref{eqs1 2}) satisfying the regularity assumption (\ref{eqs2 1}). Then one has
\begin{equation}\label{eqs5 14}
    \left| (R_t^\alpha)^n \right|\leq \left\{ \begin{aligned}
	&C_1 \sum_{l=1}^i{\rho ^{\left( l-1 \right) \alpha }\left( n-\left( l-1 \right) N \right) ^{ \left(l-2 \right) \alpha -1 }}, &0<\alpha \le \frac{1}{2},\\
	&C_2 \sum_{l=1}^i{\rho ^{ \left( l-1 \right) \alpha }\left( n-\left( l-1 \right) N \right) ^{l\alpha-2 }}, &\frac{1}{2}<\alpha <1,\\
\end{aligned} \right.
\end{equation}
where $(i-1)N+1\leq n\le iN,~~1\leq i\leq K.$
\end{lem}
\begin{proof}
It follows from \cite[Lemma 5.3]{Bu2024Finite} that, for the uniform mesh, the local truncation error $\left| (R_t^\alpha)^n \right|$ can be bounded as
$$
\begin{aligned}
   \left| (R_t^\alpha)^n \right| &\leq C \sum_{l=1}^{i}\rho^{\min\{l\alpha+1, 2-\alpha\}}(t_n-(l-1)\tau)^{(l-1)\alpha-\min\{l\alpha+1, 2-\alpha\}} \\
   &\leq C\sum_{l=1}^{i} \rho^{(l-1)\alpha}(n-(l-1)N)^{(l-1)\alpha} (n-(l-1)N)^{-\min\{\alpha+1, 2-\alpha\}}.
\end{aligned}
$$
Thus, for $0<\alpha \le \frac{1}{2}$, it leads to
$$
\begin{aligned}
   \left| (R_t^\alpha)^n \right| & \leq C_1\sum_{l=1}^{i} \rho^{(l-1)\alpha}(n-(l-1)N)^{(l-1)\alpha} (n-(l-1)N)^{-(\alpha+1)} \\
   &\leq  C_1 \sum_{l=1}^i{\rho ^{\left( l-1 \right) \alpha }\left( n-\left( l-1 \right) N \right) ^{ \left(l-2 \right) \alpha -1 }},
\end{aligned}
$$
and for $ \frac{1}{2}<\alpha <1$, it yields
$$
\begin{aligned}
   \left| (R_t^\alpha)^n \right| & \leq C_2\sum_{l=1}^{i} \rho^{(l-1)\alpha}(n-(l-1)N)^{(l-1)\alpha} (n-(l-1)N)^{-(2-\alpha)} \\
   &\leq  C_2 \sum_{l=1}^i{\rho ^{ \left( l-1 \right) \alpha }\left( n-\left( l-1 \right) N \right) ^{l\alpha-2 }}.
\end{aligned}
$$
This completes the proof.
\end{proof}

Next we derive the numerical scheme.
By (\ref{eqs3 1}), we can obtain the semi-discrete scheme of (\ref{eqs1 1})--(\ref{eqs1 2}) as follows: find $U^n\in H^1_0(\Omega),\ n=1,2,\cdots,KN$ such that
\begin{equation}\label{eqsa20240601}
  (D^\alpha_N U^n,v) + B(U^n,v)=  b(U^{n-N}, v) + (f^n,v), \forall v\in H^1_0(\Omega),
\end{equation}
where $U^n=\phi(x,t_n),n=-N,-N+1,\cdots,0$ $(u,v):=\int_0^Luvdx$, and $B(u,v):=p\left(\frac{d u}{dx},\frac{dv}{dx}\right) - (a(x)u,v).$ Define $\|u\|_l=\|u\|_{H^l(0,L)},\ l\geq0$.  It is obvious that the bilinear form $B(u,v)$ satisfying
\begin{equation}\label{eq240606a}
B(u,u)\geq C_1\|u\|_1^2,\ B(u,v)\leq C_2\|u\|_1\|v\|_1.
\end{equation}
Let $X_h$ be a continuous piecewise linear finite element space in $H_0^1(0,L)$ based on quasi-uniform mesh with maximum diameter $h$. We propose the fully discrete finite element scheme of (\ref{eqs1 1})--(\ref{eqs1 2}): find $U^n_h\in X_h$ satisfying
\begin{equation}\label{eqs3 7}
\left(D^{\alpha}_NU_h^n,v\right)+B\big(U_h^n,v\big)=b\big(U_h^{n-N},v\big)+\big(f^n,v\big),\quad \forall v \in X_h,
\end{equation}
 where $U^{n}_h \in X_h, n=-N,-N+1,\cdots,0$ are suitable approximation of $\phi(x,t_n)$.  
%


\section{Discrete fractional Gr\"onwall inequality with constant delay term}
In this section, we develop a novel Gr\"onwall inequality which will be used to the local convergence analysis of the finite element scheme (\ref{eqs3 7}). Let
$$
K_{\beta,n}=\begin{cases}1+\frac{1-n^{1-\beta}}{\beta-1},&\beta\ne1,\\1+\ln n,&\beta=1.\end{cases}
$$
 Define a family of discrete coefficients $P_{l}, l=0,1,\cdots$ recursively by
\begin{equation}\label{eqs4 1}
P_{0}=\frac{1}{a_{0}},\ P_{n-k}=\frac{1}{a_{0}}\sum_{j=k+1}^n{P_{n-j}}\left( a_{j-\left( k+1 \right)}-a_{j-k} \right) ,\ 1\leq  k\leq  n-1.
\end{equation}
For these coefficients, there are several useful results as follows. In view of the fact that there will be too many estimation formulas in this section, in order to avoid too many $C$ in the derivation, we denote $y\lesssim z$ as $y \leq Cz$ here.
\begin{lem}\label{jullemma5.3}\cite{Li2021Sharp} For the discrete coefficient $ P_{l}$ defined in (\ref{eqs4 1}), then the following results hold
\begin{equation*}
    \begin{aligned}
    &(\romannumeral1)\quad\sum_{j=k}^n P_{n-j}a_{j-k}=1,\quad 1\leq k\leq n;\\
    &(\romannumeral2)\quad P_l<\Gamma(2-\alpha)\rho^\alpha\left(l+1\right)^{\alpha-1},\quad 0\leq l \leq n;\\
    &(\romannumeral3)\quad\ \sum_{j=1}^n{P_{n-j}\leq \frac{t_{n}^{\alpha}}{\Gamma \left( 1+\alpha \right)}}.
    \end{aligned}
\end{equation*}
\end{lem}

\begin{lem}\label{jullemma5.4}\cite{Chen2021Blow} Assume that $\beta$ be a non-negative constant. Then
$$
\sum_{j=1}^n j^{-\beta}\lesssim  K_{\beta,n},\ n=1,2,\cdots.
$$
Moreover, $K_{\beta,n}$  is a strictly decreasing function about $\beta$ and one has $\operatorname*{lim}_{\beta\rightarrow1}K_{\beta,n}=K_{1,n}$ for each $n$.
\end{lem}
\begin{lem}\label{jullemma5.5} For the discrete coefficient $ P_{l}$ defined in (\ref{eqs4 1}), we have
\begin{equation}\label{eqs5 15}
	\sum_{j=qN+1}^n{P_{n-j}\left( j-q N \right) ^{-\beta}}\lesssim \Gamma(2-\alpha) \rho ^{\alpha}{K_{\beta ,n}\left( n-q N \right) ^{\alpha -1}},\ \ \beta \ge 1,\\
\end{equation}
\begin{equation}
    \frac{1}{\Gamma \left( 1-\beta \right)}\sum_{j=qN+1}^n{P_{n-j}\left( j-q N \right) ^{-\beta}}\lesssim \frac{\Gamma(2-\alpha)\rho ^{\alpha}}{\Gamma \left( 1-\beta +\alpha \right)}{\left( n-q N \right) ^{\alpha -\beta}},\ \ \beta <1,
\end{equation}\label{eqs5 16}
where $q=0,1,2,\cdots i-1$, $ (i-1)N+1\leq n \leq iN$ and $1\leq i\leq K$.
\end{lem}
\begin{proof}
For $\beta\ge 1$, it follows from Lemma \ref{jullemma5.3} that
$$
\begin{small}
 \begin{aligned}
  \sum_{j=qN+1}^n{P_{n-j}\left( j-q N \right) ^{-\beta}}\leq &\Gamma(2-\alpha)\rho ^{\alpha} \sum_{j=qN+1}^{n}\left( j-q N \right) ^{-\beta}(n-j+1)^{\alpha-1}\\
\leq &\Gamma(2-\alpha)\rho ^{\alpha} \sum_{i=1}^{n-qN}i^{-\beta} (n-i-qN+1)^{\alpha-1} \\
\leq &\Gamma(2-\alpha)\rho ^{\alpha}\left[\left(\frac{n-qN}{2}\right)^{\alpha-1} \sum_{i=1}^{\lceil \frac{n-qN}{2} \rceil}i^{-\beta}\right.\\
  &\left.+\left(\frac{n-qN}{2}\right)^{-\beta} \sum_{i=\lceil \frac{n-qN}{2} \rceil+1}^{n-qN}(n-i-qN+1)^{\alpha-1}\right],
\end{aligned}
\end{small}
$$
where  $\lceil y \rceil$ denotes an integer satisfying $y\leq \lceil y \rceil \leq y+1$.
Notice that $i-1\leq s\leq i\leq n-qN$ implies $(n-i-qN+1)^{\alpha-1}\leq (n-s-qN)^{\alpha-1}$. Then
\begin{equation}\label{eqs20240605a}
  \begin{aligned}
  \left(\frac{n-qN}{2}\right)^{-\beta}\sum_{i=\lceil \frac{n-qN}{2} \rceil+1}^{n-qN}(n-i-qN+1)^{\alpha-1}&\leq \left(\frac{n-qN}{2}\right)^{-\beta}\sum_{i=\lceil \frac{n-qN}{2} \rceil+1}^{n-qN}\int_{i-1}^{i}{(n-s-qN)^{\alpha-1}}ds \\
  &\leq \left(\frac{n-qN}{2}\right)^{-\beta} \int_{\frac{n-qN}{2}}^{n-qN}{(n-s-qN)^{\alpha-1}}ds \\
  &=\frac{1}{\alpha}\left(\frac{n-qN}{2}\right)^{\alpha-\beta}.
\end{aligned}
\end{equation}
Combining with Lemma \ref{jullemma5.4}, it yields
$$
\sum_{j=qN+1}^n{P_{n-j}\left( j-q N \right) ^{-\beta}}\lesssim \Gamma(2-\alpha) \rho ^{\alpha}{K_{\beta ,n}\left( n-q N \right) ^{\alpha -1}},\ \ \beta \ge 1.
$$

For $\beta<1$,  Lemma \ref{jullemma5.3} gives
$$
\begin{aligned}
  \frac{1}{\Gamma \left( 1-\beta \right)}\sum_{j=qN+1}^n{P_{n-j}\left( j-q N \right) ^{-\beta}} &\leq \frac{\Gamma(2-\alpha)\rho^{\alpha}}{\Gamma(1-\beta)}\sum_{j=qN+1}^{n}\left( j-q N \right) ^{-\beta}(n-j+1)^{\alpha-1}\\
  & \leq \frac{\Gamma(2-\alpha) \rho^\alpha}{\Gamma(1-\beta)}\left(\sum_{i=1}^{n-qN-1}i^{-\beta} (n-i-qN+1)^{\alpha-1}+ (n-qN)^{-\beta}\right).
  \end{aligned}
$$
Similar to the discussion of (\ref{eqs20240605a}), we have
$$
\begin{aligned}
  \frac{1}{\Gamma \left( 1-\beta \right)}\sum_{j=qN+1}^n{P_{n-j}\left( j-q N \right) ^{-\beta}} &\leq \frac{\Gamma(2-\alpha) \rho^\alpha}{\Gamma(1-\beta)}\left(\sum_{i=1}^{n-qN-1}\int_{i-1}^{i}{s^{-\beta} (n-s-qN)^{\alpha-1}}ds+ (n-qN)^{-\beta}\right)\\
  &\leq \frac{\Gamma(2-\alpha) \rho^\alpha}{\Gamma(1-\beta)}\left(\int_{0}^{n-qN}{s^{-\beta} (n-s-qN)^{\alpha-1}}ds+ (n-qN)^{-\beta}\right)\\
  &\leq \frac{\Gamma(2-\alpha) \rho^\alpha}{\Gamma(1-\beta)}\left(\mathcal{B}(1-\beta, \alpha)(n-qN)^{\alpha-\beta}+ (n-qN)^{-\beta}\right)\\
  &\lesssim \frac{\Gamma(2-\alpha) \rho^\alpha}{\Gamma(1-\beta+\alpha)}(n-qN)^{\alpha-\beta},
\end{aligned}
$$
where $\mathcal{B}(\cdot,\cdot)$ represents the Beta function.
\end{proof}

 Let $Z_1=\left( R_n^{(0)},R_{n-1}^{(0)},\cdots ,R_1^{(0)} \right)^T $, $Z_2=\left( 1,1,\cdots ,1 \right)^T\in\mathds{R}^n $ and
\begin{equation}\label{eqs5 17}
    J=\begin{bmatrix}0&\cdots&0&P_0&\cdots&P_{n-N-2}&P_{n-N-1}\\0&\cdots&0&0&\cdots&P_{n-N-3}&P_{n-N-2}\\\vdots&&\vdots&\vdots&\ddots&\vdots&\vdots\\0&\cdots&0&0&\cdots&0&P_0\\0&\cdots&0&0&\cdots&0&0\\\vdots&&\vdots&\vdots&&\vdots&\vdots\\0&\cdots&0&0&\cdots&0&0\end{bmatrix}_{n\times n},
\end{equation}
where $R_k^{(0)}=\sum_{j=1}^{k}P_{k-j}\left| (R_t^\alpha)^j \right|,k=1,2,\cdots,n$, $(i-1)N+1\leq n \leq iN$ and $1\leq i\leq K$. In fact, by observing the upper triangular matrix $J$, it is easy to check that $J^q=0$ for  $q\geq i$. Next, we can obtain some results on $P_l$ further.

\begin{lem}\label{jullemma5.6} For the first element $R_n^{(0)}$ of $Z_1$, if $0<\alpha \leq \frac{1}{2}$, then
\begin{equation}\label{eqs5 18}
R_n^{(0)}\lesssim \left\{ \begin{aligned}
	&\Gamma \left( 2-\alpha \right) K_{\beta _1,n}\rho ^{\alpha}n^{\alpha -1},\qquad \qquad \qquad \qquad \qquad \qquad \qquad \quad \ i=1,\\
	&\Gamma \left( 2-\alpha \right) K_{\beta _1,n}\rho ^{\alpha}n^{\alpha -1}+\Gamma \left( 2-\alpha \right) K_{\beta _2,n}\rho ^{2\alpha}\left( n-N \right) ^{\alpha -1},\qquad i=2,\\
	&\begin{aligned}
	&\Gamma \left( 2-\alpha \right) K_{\beta _1,n}\rho ^{\alpha}n^{\alpha -1}+\Gamma \left( 2-\alpha \right) K_{\beta _2,n}\rho ^{2\alpha}\left( n-N \right) ^{\alpha -1}  \\
	&+\Gamma \left( 2-\alpha \right) \sum_{l=3}^i{\frac{\Gamma \left( \left( l-2 \right) \alpha \right) \rho ^{l\alpha}\left( n-\left( l-1 \right) N \right) ^{\left( l-1 \right) \alpha -1}}{\Gamma \left( \left( l-1 \right) \alpha \right)}},\qquad i\geq 3;
\end{aligned}
\end{aligned} \right.
\end{equation}
if $\frac{1}{2}<\alpha<1$, then
\begin{equation}\label{eqs5 19}
R_n^{(0)}\lesssim \left\{ \begin{aligned}
        &\Gamma \left( 2-\alpha \right) K_{\beta _3,n}\rho ^{\alpha} n^{\alpha -1},&i=1,\\
        &\Gamma \left( 2-\alpha \right) K_{\beta _3,n}\rho ^{\alpha} n^{\alpha -1}+ \Gamma \left( 2-\alpha \right)\sum_{l=2}^{i}{\frac{\rho^{l\alpha}\Gamma(l\alpha-1)}{\Gamma((l+1)\alpha-1)}(n-(l-1)N)^{(l+1)\alpha-2}},&i\geq 2,
\end{aligned} \right.
\end{equation}
where $\beta_1=1+\alpha, \beta_2=1, \beta_3=2-\alpha$.\\
\end{lem}
\begin{proof}
  First, for $0<\alpha \leq \frac{1}{2}$ and $i=1$, it follows from the Lemma \ref{jullemma5.2} and the Lemma \ref{jullemma5.5} that
\begin{equation}\label{eqs20240605b}
  \begin{aligned}
  R_n^{(0)}&=\sum_{j=1}^n{P_{n-j}}\left|(R_t^\alpha)^j\right|\\
  &\lesssim \sum_{j=1}^n{P_{n-j}} j^{-\alpha -1}\lesssim \Gamma \left( 2-\alpha \right) K_{\beta _1,n}\rho ^{\alpha}n^{\alpha -1}.
\end{aligned}
\end{equation}
For $i=2$, splitting $R_n^{(0)}$ into two parts and using the Lemma \ref{jullemma5.2} and the Lemma \ref{jullemma5.5}, we have
\begin{equation}\label{eqs20240605c}
\begin{aligned}
  R_n^{(0)}=&{\sum_{j=1}^{N}{P_{n-j}\left|(R_t^\alpha)^j\right|}}+\sum_{j= N +1}^n{P_{n-j}}\left|(R_t^\alpha)^j\right|\\
\lesssim &{\sum_{j=1}^{N}{P_{n-j}\sum_{m=1}^1 {\rho ^{\left( m-1 \right) \alpha}\left( j-\left( m-1 \right) N \right) ^{\left( m-2 \right) \alpha -1}}}}\\
  &+\sum_{j= N+1}^n{P_{n-j}}\sum_{m=1}^{2}{\rho ^{\left( m-1 \right) \alpha}\left( j-\left( m-1 \right) N \right) ^{\left( m-2 \right) \alpha -1}}\\
=&\sum_{j=1}^n{P_{n-j}} j^{-\alpha -1}+\rho^\alpha \sum_{j=N+1}^n{P_{n-j}}\left( j-N \right) ^{-1} \\
\lesssim & \Gamma \left( 2-\alpha \right) K_{\beta _1,n}\rho ^{\alpha}n^{\alpha -1}+ \Gamma \left( 2-\alpha \right) K_{\beta _2,n}\rho ^{2\alpha}\left( n-N \right) ^{\alpha -1}.
\end{aligned}
\end{equation}
Similarly, splitting $R_n^{(0)}$ into $i$ parts when $i\geq 3$, it is easy to know that
 $$
\begin{aligned}
  R_n^{(0)}=&\sum_{l=1}^{i-1}{\sum_{j=\left( l-1 \right) N+1}^{lN}{P_{n-j}\left|(R_t^\alpha)^j\right|}}+\sum_{j=(i-1) N +1}^n{P_{n-j}}\left|(R_t^\alpha)^j\right|\\
\lesssim &\sum_{l=1}^{i-1}{\sum_{j=\left( l-1 \right) N+1}^{lN}{P_{n-j}\sum_{m=1}^l {\rho ^{\left( m-1 \right) \alpha}\left( j-\left( m-1 \right) N \right) ^{\left( m-2 \right) \alpha -1}}}}\\
&+\sum_{j=(i-1) N+1}^n{P_{n-j}}\sum_{m=1}^{i}{\rho ^{\left( m-1 \right) \alpha}\left( j-\left( m-1 \right) N \right) ^{\left( m-2 \right) \alpha -1}}\\
=&\sum_{j=1}^n{P_{n-j}} j^{-\alpha -1}+\rho^\alpha \sum_{j=N+1}^n{P_{n-j}}\left( j-N \right) ^{-1}+\cdots \\
&+\rho ^{(i-1) \alpha}\sum_{j=(i-1) N+1}^n{P_{n-j}}\left( j-(i-1) N \right) ^{\left( i -2 \right) \alpha -1}\\
\end{aligned}
$$
For $l\geq 3$, Lemma \ref{jullemma5.5} gives
$$
\begin{aligned}
   \rho ^{(l-1) \alpha}\sum_{j=(l-1) N+1}^n{P_{n-j}}\left( j-(l-1) N \right) ^{\left( l -2 \right) \alpha -1}\lesssim \Gamma \left( 2-\alpha \right)\frac{\Gamma \left( \left( l-2 \right) \alpha \right) \rho ^{l\alpha}\left( n-\left( l-1 \right) N \right) ^{\left( l-1 \right) \alpha -1}}{\Gamma \left( \left( l-1 \right) \alpha \right)}.
\end{aligned}
$$
So
\begin{equation}\label{eqs20240605d}
\begin{aligned}
  R_n^{(0)} \lesssim &\Gamma \left( 2-\alpha \right) K_{\beta _1,n}\rho ^{\alpha}n^{\alpha -1}+\Gamma \left( 2-\alpha \right) K_{\beta _2,n}\rho ^{2\alpha}\left( n-N \right) ^{\alpha -1}\\
  &+\Gamma \left( 2-\alpha \right) \sum_{l=3}^{i}{\frac{\Gamma \left( \left( l-2 \right) \alpha \right) \rho ^{l\alpha}\left( n-\left( l-1 \right) N \right) ^{\left( l-1 \right) \alpha -1}}{\Gamma \left( \left( l-1 \right) \alpha \right)}},\ i\geq 3.
\end{aligned}
\end{equation}
Therefore, the combination of (\ref{eqs20240605b}), (\ref{eqs20240605c}) and (\ref{eqs20240605d}) yields (\ref{eqs5 18}). In addition, it is clear that, according to a similar derivation, we can obtain (\ref{eqs5 19}) (i.e. the case $\frac12<\alpha \leq 1$) only needing to pay attention to $2 -l\alpha < 1, l\geq 2$ and a trivial change in (\ref{eqs5 14}).
\end{proof}

\begin{lem}\label{jullemma20240605a}For $(i-1)N+1\leq n\leq iN$,  $ 1\leq i\leq K$ and the first element $R_{n-qN}^{(q)}$ of $J^q Z_1,q=0,1,\cdots,i-1$, if $0<\alpha \leq \frac{1}{2}$, then the following three cases hold
\begin{equation}\label{eqs5 20}
    \begin{aligned}
R_{n-qN}^{(q)} \lesssim &\left( \Gamma \left( 2-\alpha \right) \right) ^{q+1} \left( \frac{\Gamma \left( \alpha \right)\rho ^{\left( q+1 \right) \alpha}K_{\beta _1,n}(n-qN)^{\left( q+1 \right) \alpha -1}}{\Gamma \left( \left( q+1 \right) \alpha \right)}\right.\\
&+\frac{\Gamma \left( \alpha \right)\rho ^{\left( q+2 \right) \alpha}K_{\beta _2,n}\left( n-(q+1)N \right) ^{\left( q+1 \right) \alpha -1}}{\Gamma \left( \left( q+1 \right) \alpha \right)}  \\
&+\left. \sum_{l=3}^{i-q }{\frac{\Gamma((l-2)\alpha)\rho ^{\left( l+q \right) \alpha}\left( n-\left( l+q-1 \right) N \right) ^{\left( l+q-1 \right) \alpha -1}}{\Gamma \left( \left( l+q-1 \right) \alpha \right)}} \right)\ \text{for}\ q=0,1,\cdots,i-3\ \text{and}\ i\geq 3;
\end{aligned}
\end{equation}
\begin{equation}\label{eqs5 21}
  \begin{aligned}
    R_{n-qN}^{(q)}\lesssim & \left(\frac{(\Gamma(2-\alpha))^{i-1}\Gamma(\alpha)\rho^{(i-1)\alpha}K_{\beta _1,n}(n-(i-2)N)^{(i-1)\alpha-1}}{\Gamma((i-1)\alpha)}\right.\\
    &\left.+\frac{(\Gamma(2-\alpha))^{i-1}\Gamma(\alpha)\rho^{i\alpha}K_{\beta _2,n}(n-(i-1)N)^{(i-1)\alpha-1}}{\Gamma((i-1)\alpha)}\right)\ \text{for}\ q= i-2\ \text{and}\  i\geq 2;
\end{aligned}
\end{equation}
and
\begin{equation}\label{eqs5 22}
    R_{n-qN}^{(q)} \lesssim \frac{(\Gamma(2-\alpha))^{i}\Gamma(\alpha)\rho^{i\alpha}K_{\beta _1,n}(n-(i-1)N)^{i\alpha-1}}{\Gamma(i\alpha)}\ \text{for}\ q=i-1.
\end{equation}
If $\frac{1}{2}<\alpha <1 $, then the following two cases hold
\begin{equation}\label{eqs5 23}
    \begin{aligned}
        R_{n-qN}^{(q)} \lesssim& \left( \Gamma \left( 2-\alpha \right) \right) ^{q+1}\left(\frac{\Gamma \left( \alpha \right)\rho ^{\left( q+1 \right) \alpha}K_{\beta _3,n}(n-qN)^{\left( q+1 \right) \alpha -1}}{\Gamma \left( \left( q+1 \right) \alpha \right)}\right.\\
        &\left.+\sum_{l=2}^{i-q}\frac{\Gamma(l\alpha-1)\rho^{(l+q)\alpha}(n-(l+q-1)N)^{(l+q+1)\alpha-2}}{\Gamma((l+q+1)\alpha-1)}\right)\ \text{for}\ q=0,1,2,\cdots,i-2\  \text{and}\ i\geq 2;
    \end{aligned}
\end{equation}
and
\begin{equation}\label{eqs5 24}
    R_{n-qN}^{(q)} \lesssim \frac{(\Gamma(2-\alpha))^{i}\Gamma(\alpha)\rho^{i\alpha}K_{\beta _3,n}(n-(i-1)N)^{i\alpha-1}}{\Gamma(i\alpha)}\ \text{for}\ q=i-1.
\end{equation}
\end{lem}
\begin{proof}
Here, we mainly discuss the case $0<\alpha \leq \frac{1}{2}$. It follows from Lemma \ref{jullemma5.6} that (\ref{eqs5 22}) for $i=1$ and (\ref{eqs5 21}) for $i=2$ are true.
For $n>N$,
we have
\begin{equation}\label{eqs5 29}
   \begin{aligned}
      JZ_1&=\left(\sum_{j=1}^{n-N}{P_{n-N-j}R_j^{(0)}},\sum_{j=1}^{n-N-1}{P_{n-N-1-j}R_j^{(0)}},\cdots,P_0 R_1^{(0)},\underset{N}{\underbrace{0,\cdots, 0}} \right)^T\\
      &:=\left(R_{n-N}^{(1)},R_{n-N-1}^{(1)},\cdots,R_{1}^{(1)},\underset{N}{\underbrace{0,\cdots, 0}}\right)^T.
  \end{aligned}
 \end{equation}
 Notice that $ 1\leq n-N\leq N$ for $i=2$, thus Lemma \ref{jullemma5.5} and Lemma \ref{jullemma5.6} give
 \begin{equation}\label{eqs20240613b}
   \begin{aligned}
    R_{n-N}^{(1)}&= \sum_{j=1}^{n-N}{P_{n-N-j}R_j^{(0)}}\\
    &\lesssim \Gamma \left( 2-\alpha \right) K_{\beta _1,n}\rho ^{\alpha}\sum_{j=1}^{n-N}P_{n-N-j} j^{\alpha-1}\\
    &\lesssim  \frac{\left( \Gamma \left( 2-\alpha \right) \right) ^2\Gamma \left( \alpha \right)\rho ^{2\alpha}K_{\beta _1,n}(n-N)^{2\alpha -1}}{\Gamma \left( 2\alpha \right)}.
  \end{aligned}
 \end{equation}
It implies that (\ref{eqs5 22}) holds for $i=2$.

Next we consider the case $i=3$. First it is obvious that Lemma \ref{jullemma5.6} leads to (\ref{eqs5 20}) for the present case. Combining with Lemma \ref{jullemma5.5} and Lemma \ref{jullemma5.6}, for $i=3$ we get
 \begin{equation}\label{eqs20240613a}
   \begin{aligned}
    R_{n-N}^{(1)}&= \sum_{j=1}^{N}{P_{n-N-j}R_j^{(0)}}+ \sum_{j=N+1 }^{n-N}{P_{n-N-j}}R_j^{(0)}\\
    &\lesssim \Gamma \left( 2-\alpha \right) K_{\beta _1,n}\rho ^{\alpha}\sum_{j=1}^{n-N}P_{n-N-j} j^{\alpha-1}+\Gamma \left( 2-\alpha \right) K_{\beta _2,n}\rho ^{2\alpha}\sum_{j=N+1}^{n-N}P_{n-N-j}(j-N)^{\alpha-1}\\
    &\lesssim \left( \Gamma \left( 2-\alpha \right) \right) ^2 \left( \frac{\Gamma \left( \alpha \right)\rho ^{2\alpha}K_{\beta _1,n}(n-N)^{2\alpha -1}}{\Gamma \left( 2\alpha \right)}+\frac{\Gamma \left( \alpha \right)\rho ^{3\alpha}K_{\beta _2,n}\left( n-2N \right) ^{2\alpha -1}}{\Gamma \left( 2\alpha \right)} \right).
    \end{aligned}
 \end{equation}
It implies that (\ref{eqs5 21}) holds for $i=3$. Notice that
$$
  \begin{aligned}
      J^2 Z_1&=\left(\sum_{j=1}^{n-2N}{P_{n-2N-j}R_j^{(1)}},\sum_{j=1}^{n-2N-1}{P_{n-2N-1-j}R_j^{(1)}},\cdots,P_0 R_1^{(1)},\underset{2N}{\underbrace{0,\cdots, 0}} \right)^T\\
      &:=\left(R_{n-2N}^{(2)},R_{n-2N-1}^{(2)},\cdots,R_{1}^{(2)},\underset{2N}{\underbrace{0,\cdots, 0}}\right)^T,
  \end{aligned}
  $$
and $1\leq n-2N\leq N$ for $i=3$, thus Lemma \ref{jullemma5.5} and (\ref{eqs5 22}) for $i=2, q=1$ give
 $$
 \begin{aligned}
   R_{n-2N}^{(2)} & = \sum_{j=1}^{n-2N}{P_{n-2N-j}R_{(j+N)-N}^{(1)}} \\
    & \lesssim \frac{\left( \Gamma \left( 2-\alpha \right) \right) ^2 \Gamma \left( \alpha \right)\rho ^{2\alpha}K_{\beta _1,n}}{\Gamma \left( 2\alpha \right)} \sum_{j=1}^{n-2N}{P_{n-2N-j}j^{2\alpha-1}}\\
    &\lesssim  \frac{\left( \Gamma \left( 2-\alpha \right) \right) ^3 \Gamma \left( \alpha \right)\rho ^{3\alpha}K_{\beta _1,n}(n-2N)^{3\alpha -1}}{\Gamma \left( 3\alpha \right)}.
 \end{aligned}
  $$
It means that (\ref{eqs5 22}) holds for $i=3$.


Now, we adopt nested mathematical induction to discuss (\ref{eqs5 20})--(\ref{eqs5 22}) for $i> 3$. First it is obvious that (\ref{eqs5 20})--(\ref{eqs5 22}) hold for $i=1,2,3$. Assume that (\ref{eqs5 20})--(\ref{eqs5 22}) are true for $i=1,2,\cdots,r$, $r\geq 3$.
Taking $i=r+1$, we prove (\ref{eqs5 20}) firstly. According to Lemma \ref{jullemma5.6}, (\ref{eqs5 20}) apparently holds for $q = 0$. Assume that it holds when $q=m, 0\leq m<r-2$. Taking $q=m+1$, it is clear that
\begin{equation}\label{eqs20240614d}
  \begin{aligned}
    J^{m+1}Z_1&=J\left( R_{n-mN}^{(m)},R_{n-mN-1}^{(m)},\cdots R_{1}^{(m)},\underset{mN}{\underbrace{0,\cdots, 0}} \right) ^T\\
    &=\left( \sum_{j=1}^{n-(m+1)N}{P_{n-(m+1)N-j}R_{j}^{(m)}},\sum_{j=1}^{n-(m+1)N-1}{P_{n-(m+1)N-1-j}R_{j}^{(m)}},\cdots,P_0 R_1^{(m)},\underset{(m+1)N}{\underbrace{0,\cdots, 0}} \right)^T\\
    &:=\left(R_{n-(m+1)N}^{(m+1)},R_{n-(m+1)N-1}^{(m+1)},\cdots,R_{1}^{(m+1)},\underset{(m+1)N}{\underbrace{0,\cdots, 0}}  \right)^T.
\end{aligned}
\end{equation}
Splitting $R_{n-(m+1)N}^{(m+1)}$ into $r-m$ parts yields
\begin{equation}\label{eqs20240614e}
\begin{aligned}
  R_{n-(m+1)N}^{(m+1)}=&\sum_{j=1}^{N}P_{n-(m+1)N-j}R_{(j+mN)-mN}^{(m)}+\sum_{j=N+1}^{2N}P_{n-(m+1)N-j}R_{(j+mN)-mN}^{(m)}+\cdots\\
&+\sum_{j=(r-m-1)N+1}^{n-(m+1)N}P_{n-(m+1)N-j}R_{(j+mN)-mN}^{(m)}.
\end{aligned}
\end{equation}
According to the assumption that (\ref{eqs5 20})--(\ref{eqs5 22}) are true for $i=1,2,\cdots, r$. It means that the boundedness of $R^{(m)}_{n-mN}$ for $0<n-mN\leq (r-m)N$ is known. By utilizing (\ref{eqs5 20})--(\ref{eqs5 22}) for $i\leq r$ and Lemma \ref{jullemma5.5}, we reach
\begin{equation}\label{eqs20240614f}
  \begin{aligned}
R_{n-(m+1)N}^{(m+1)} \lesssim& \left( \Gamma \left( 2-\alpha \right) \right) ^{m+1}\left(\frac{\Gamma \left( \alpha \right) K_{\beta _1,n}\rho ^{\left( m+1 \right) \alpha}}{\Gamma \left( \left( m+1 \right) \alpha \right)}\sum_{j=1}^{n-(m+1)N}{P_{n-(m+1)N-j}}j^{\left( m+1 \right) \alpha -1}\right.\\
&+\frac{ \Gamma \left( \alpha \right) K_{\beta _2,n}\rho ^{\left( m+2 \right) \alpha}}{\Gamma \left( \left( m+1 \right) \alpha \right)}\sum_{j=N+1}^{n-(m+1)N}{P_{n-(m+1)N-j}}\left( j-N \right) ^{\left( m+1 \right) \alpha -1}+\cdots\\
&\left. +\frac{ \Gamma((r-m-2)\alpha)\rho ^{r \alpha}}{\Gamma \left( (r-1) \alpha \right)}\sum_{j=(r-m-1) N+1}^{n-(m+1)N}{P_{n-(m+1)N-j}}\left( j-(r-m-1)N \right) ^{(r-1) \alpha -1}\right)\\
\lesssim &\left( \Gamma \left( 2-\alpha \right) \right) ^{m+2} \left( \frac{\Gamma \left( \alpha \right)\rho ^{\left( m+2 \right) \alpha}K_{\beta _1,n}(n-(m+1)N)^{\left( m+2 \right) \alpha -1}}{\Gamma \left( \left( m+2 \right) \alpha \right)} \right. \\
&+\frac{\Gamma \left( \alpha \right)\rho ^{\left( m+3 \right) \alpha}K_{\beta _2,n}\left( n-(m+2)N \right) ^{\left( m+2 \right) \alpha -1}}{\Gamma \left( \left( m+2 \right) \alpha \right)}\\
&+\left. \sum_{l=3}^{r+1 -(m+1)}{\frac{\Gamma((l-2)\alpha)\rho ^{\left( l+m+1 \right) \alpha}\left(n-(l+m)N \right) ^{\left( l+m \right) \alpha -1}}{\Gamma \left( \left( l+m \right) \alpha \right)}} \right).
\end{aligned}
\end{equation}
The above derivation indicates that (\ref{eqs5 20}) is true for $q=m+1$. Thus by combining with (\ref{eqs20240614d}), (\ref{eqs20240614e}) and (\ref{eqs20240614f}), the mathematical induction on $q$ immediately leads to (\ref{eqs5 20}) for $i=r+1$ .

We notice that
$$
\begin{aligned}
    J^{r-1}Z_1&=J\left( J^{r-2} Z_1\right)\\
    &=\left( \sum_{j=1}^{n-(r-1)N}{P_{n-(r-1)N-j}R_{j}^{(r-2)}},\sum_{j=1}^{n-(r-1)N-1}{P_{n-(r-1)N-1-j}R_{j}^{(r-2)}},\cdots,P_0 R_1^{(r-2)},\underset{(r-1)N}{\underbrace{0,\cdots, 0}} \right)^T\\
    &:=\left(R_{n-(r-1)N}^{(r-1)},R_{n-(r-1)N-1}^{(r-1)},\cdots,R_{1}^{(r-1)},\underset{(r-1)N}{\underbrace{0,\cdots, 0}}  \right)^T.
\end{aligned}$$
Due to $N+1\leq n-(r-1)N \leq 2N$, applying Lemma \ref{jullemma5.5}, (\ref{eqs5 22}) for $i=r-1, q=r-2$ and (\ref{eqs5 21}) for $i=r, q=r-2$, it yields
$$
\begin{aligned}
    R_{n-(r-1)N}^{(r-1)}=&\sum_{j=1}^N {P_{n-(r-1)N-j}R_{(j+(r-2)N)-(r-2)N}^{(r-2)}}+\sum_{j=N+1}^{n-(r-1)N} {P_{n-(r-1)N-j}R_{(j+(r-2)N)-(r-2)N}^{(r-2)}}\\
    \lesssim&  \frac{\left( \Gamma \left( 2-\alpha \right) \right) ^{r-1}\Gamma \left( \alpha \right)\rho ^{\left( r-1 \right) \alpha}K_{\beta _1,n}}{\Gamma \left( \left( r-1 \right) \alpha \right)}\sum_{j=1}^{n-(r-1)N} {P_{n-(r-1)N-j} j^{\left( r-1 \right) \alpha -1}}\\
    &+\frac{\left( \Gamma \left( 2-\alpha \right) \right) ^{r-1}\Gamma \left( \alpha \right)\rho ^{ r  \alpha}K_{\beta _2,n}}{\Gamma \left( \left( r-1 \right) \alpha \right)}\sum_{j=N+1}^{n-(r-1)N} {P_{n-(r-1)N-j} (j-N)^{\left( r-1 \right) \alpha -1}}\\
    \lesssim & \left(\frac{(\Gamma(2-\alpha))^{r}\Gamma(\alpha)\rho^{r\alpha}K_{\beta _1,n}(n-(r-1)N)^{r\alpha-1}}{\Gamma(r\alpha)}\right.\\
    &\left.+\frac{(\Gamma(2-\alpha))^{r}\Gamma(\alpha)\rho^{(r+1)\alpha}K_{\beta _2,n}(n-rN)^{r\alpha-1}}{\Gamma(r\alpha)}\right).
\end{aligned}
$$
Hence we obtain (\ref{eqs5 21}) for $i=r+1$.
Besides, it is clear that a similar derivation can give (\ref{eqs5 22}) for $i= r+1$.
Therefore, the above derivation shows that (\ref{eqs5 20})--(\ref{eqs5 22}) are true for $i=r+1$. Now, we can conclude that (\ref{eqs5 20})--(\ref{eqs5 22}) hold for $0<\alpha \leq \frac{1}{2}$ by the mathematical induction on $i$. Actually, doing the discussion like the case $0<\alpha \leq \frac{1}{2}$, we can also deduce that (\ref{eqs5 23})--(\ref{eqs5 24}) is true for $\frac{1}{2}<\alpha <1$.
\end{proof}

\begin{lem}\label{jullemma20240605b} For $(i-1)N+1\leq n\leq iN$,  $ 1\leq i\leq K$ and the first element $R_{n-qN}^{(q)}$ of $J^q Z_1,q=0,1,\cdots,i-1$,
if $0<\alpha \leq \frac{1}{2}$, then
 \begin{equation}\label{eqs5 27}
     \sum_{q=0}^{i-1}{R_{n-qN}^{(q)}} \lesssim \left\{ \begin{aligned}
	&K_{\beta _1,n}\rho t_n^{\alpha-1},&i=1,\\
    &K_{\beta _1,n}\rho+K_{\beta _2,n}\rho^{1+\alpha}t_{n-N}^{\alpha-1}+K_{\beta _1,n}\rho t_{n-N}^{2\alpha-1},&i=2,\\
	&K_{\beta _1,n}\rho+K_{\beta _2,n}\rho^{1+\alpha} + \rho^{1+\alpha}t_{n-(i-1)N}^{(i-1)\alpha-1} +K_{\beta _2,n}\rho^{1+\alpha}t_{n-(i-1)N}^{(i-1)\alpha-1}+K_{\beta _1,n}\rho t_{n-(i-1)N}^{i\alpha-1},&i\geq 3.
\end{aligned} \right.
 \end{equation}
If $\frac{1}{2}<\alpha <1 $, then
 \begin{equation}\label{eqs5 28}
     \sum_{q=0}^{i-1}{R_{n-qN}^{(q)}} \lesssim \left\{ \begin{aligned}
	&K_{\beta _3,n}\rho t_n^{\alpha-1} ,&i=1,\\
    &K_{\beta _3,n}\rho+\rho^{2-\alpha} t_{n-N}^{3\alpha-2}, &i=2,\\
	&K_{\beta _3,n}\rho+\rho^{2-\alpha},&i\geq 3.
\end{aligned} \right.
 \end{equation}
\end{lem}
\begin{proof}
For $0<\alpha \leq \frac{1}{2} $, it follows from Lemma \ref{jullemma20240605a} and the definition of $K_{\beta,n}$ that we have the following three cases
$$
\begin{aligned}
    R_{n-qN}^{(q)}
    &\lesssim K_{\beta _1,n}\rho t_{n-(i-1)N}^{i\alpha-1},\ \text{for}\ q=i-1;
\end{aligned}
$$
$$
\begin{aligned}
    R_{n-qN}^{(q)}
    \lesssim& K_{\beta _1,n}\rho t_{n-(i-2)N}^{(i-1)\alpha-1}+K_{\beta _2,n}\rho^{1+\alpha}t_{n-(i-1)N}^{(i-1)\alpha-1},\ \text{for}\ q=i-2,\  i\geq 2
\end{aligned}
$$
and
$$
\begin{aligned}R_{n-qN}^{(q)}
&\lesssim K_{\beta _1,n}\rho t_{n-qN}^{(q+1)\alpha-1}+K_{\beta _2,n}\rho^{1+\alpha}t_{n-(q+1)N}^{(q+1)\alpha-1}+\sum_{l=3}^{i-q-1}\rho^{1+\alpha}t_{n-(l-q-1)N}^{(l+q-1)\alpha-1}+\rho^{1+\alpha}t_{n-(i-1)N}^{(i-1)\alpha-1},\ \text{for}\  q=0,1,\cdots,i-3,\  i\geq 3.
 \end{aligned}
$$
Therefore
$$
\sum_{q=0}^{i-1}{R_{n-qN}^{(q)}} \lesssim \left\{ \begin{aligned}
	&K_{\beta _1,n}\rho t_n^{\alpha-1},&i=1,\\
    &K_{\beta _1,n}\rho+K_{\beta _2,n}\rho^{1+\alpha}t_{n-N}^{\alpha-1}+K_{\beta _1,n}\rho t_{n-N}^{2\alpha-1},&i=2,\\
	&K_{\beta _1,n}\rho+K_{\beta _2,n}\rho^{1+\alpha} + \rho^{1+\alpha}t_{n-(i-1)N}^{(i-1)\alpha-1} +K_{\beta _2,n}\rho^{1+\alpha}t_{n-(i-1)N}^{(i-1)\alpha-1}+K_{\beta _1,n}\rho t_{n-(i-1)N}^{i\alpha-1},&i\geq 3.
\end{aligned} \right.
$$
In addition, applying the analogous derivation to the case $0<\alpha \leq \frac{1}{2}$, it is clear that we can obtain (\ref{eqs5 28}) easily.
\end{proof}

\begin{lem}\label{jullemma20240605c}
Assume that $(i-1)N+1\leq n\leq iN$ and $ 1\leq i\leq K$. Then $J^q Z_2$ can be bounded by
 \begin{equation}\label{eqs5 25}
     J^q Z_2\lesssim \frac{(\Gamma(2-\alpha))^{q-1}}{\Gamma(q\alpha+1)}\left(t_{n-qN}^{q\alpha},t_{n-qN-1}^{q\alpha},\cdots,t_{1}^{q\alpha},\underset{qN}{\underbrace{0,\cdots, 0}}\right)^T,\ q=1,2,\cdots,i-1.
 \end{equation}
Let the first element of $J^q Z_2$ be $\epsilon_{n-qN}^{(q)}$ and $\lambda$ be a non-negative constant. Then satisfies
 \begin{equation}\label{eqs5 26}
     \sum_{q=0}^{i-1} \lambda^q \epsilon_{n-qN}^{(q)}\lesssim 1+\frac{\lambda}{\Gamma(\alpha+1)}t_{n-N}^{\alpha}+\cdots+\frac{\lambda^{i-1}(\Gamma(2-\alpha))^{i-2}}{\Gamma((i-1)\alpha+1)}t_{n-(i-1)N}^{(i-1)\alpha}.
 \end{equation}
\end{lem}
\begin{proof}
  First the Lemma \ref{jullemma5.3} gives
$$
  \begin{aligned}
      JZ_2&=\left(\sum_{j=1}^{n-N}{P_{n-N-j}},\sum_{j=1}^{n-N-1}{P_{n-N-1-j}},\cdots,P_0 ,\underset{N}{\underbrace{0,\cdots, 0}} \right)^T\\
      &\leq \frac{1}{\Gamma(\alpha+1)}\left(t_{n-N}^{\alpha},t_{n-N-1}^{\alpha},\cdots,t_{1}^{\alpha},\underset{N}{\underbrace{0,\cdots, 0}}\right)^T.
  \end{aligned}
  $$
It implies that (\ref{eqs5 25}) holds for $q=1$. Suppose that (\ref{eqs5 25})  is true when $q=m$, i.e.,
$$
  \begin{aligned}
      J^{m}Z_2&\lesssim \frac{(\Gamma(2-\alpha))^{m-1}}{\Gamma(m\alpha+1)}\left(t_{n-mN}^{m\alpha},t_{n-mN-1}^{m\alpha},\cdots,t_{1}^{m\alpha},\underset{mN}{\underbrace{0,\cdots, 0}}\right)^T\\
      &:=\frac{(\Gamma(2-\alpha))^{m-1}}{\Gamma(m\alpha+1)}\left(\epsilon_{n-mN}^{(m)},\epsilon_{n-mN-1}^{(m)},\cdots,\epsilon_{1}^{(m)},\underset{mN}{\underbrace{0,\cdots, 0}}\right)^T.
  \end{aligned}
  $$
Then, for $q=m+1$ we have
$$
\begin{aligned}
    J^{m+1}Z_2&=J(J^m Z_2)\\
    &\lesssim \frac{(\Gamma(2-\alpha))^{m-1}}{\Gamma(m\alpha+1)}J\left(\epsilon_{n-mN}^{(m)},\epsilon_{n-mN-1}^{(m)},\cdots,\epsilon_{1}^{(m)},\underset{mN}{\underbrace{0,\cdots , 0}}\right)^T\\
    &=\frac{(\Gamma(2-\alpha))^{m-1}}{\Gamma(m\alpha+1)}\left(\sum_{j=1}^{n-(m+1)N}P_{n-(m+1)N-j}\epsilon_{j}^{(m)},\cdots,P_0\epsilon_{1}^{(m)},\underset{(m+1)N}{\underbrace{0,\cdots, 0}}\right)^T.
\end{aligned}
$$
Since Lemma \ref{jullemma5.5} implies that
$$
\begin{aligned}
    \sum_{j=1}^{n-(m+1)N}P_{n-(m+1)N-j}\epsilon_{j}^{(m)}&\lesssim \frac{\Gamma(2-\alpha)\Gamma(1+m\alpha)\rho^{(m+1)\alpha}}{\Gamma(1+(m+1)\alpha)}(n-(m+1)N)^{(m+1)\alpha}\\
    &=\frac{\Gamma(2-\alpha)\Gamma(1+m\alpha)}{\Gamma(1+(m+1)\alpha)}t_{n-(m+1)N}^{(m+1)\alpha},
\end{aligned}
$$
we reach
$$
 J^{m+1}Z_2\lesssim \frac{(\Gamma(2-\alpha))^{m}}{\Gamma((m+1)\alpha+1)}\left(t_{n-(m+1)N}^{(m+1)\alpha},t_{n-(m+1)N-1}^{(m+1)\alpha},\cdots,t_{1}^{(m+1)\alpha},\underset{(m+1)N}{\underbrace{0,\cdots, 0}}\right)^T.
$$
Therefore, the mathematical induction indicates that (\ref{eqs5 25}) is true and
\begin{equation}\label{eqs20240606a}
  \epsilon_{n-qN}^{(q)}\lesssim \frac{(\Gamma(2-\alpha))^{q-1}}{\Gamma(q\alpha+1)}t_{n-qN}^{q\alpha},\ q=1,2,\cdots,i-1.
\end{equation}
 Summing on $\epsilon_{n-qN}^{(q)}$ from $q=0$ to $i-1$, it follows from (\ref{eqs20240606a}) that (\ref{eqs5 26}) can be obtained immediately, where $\epsilon_{n}^{(0)}=1$  is used.
\end{proof}

 Now we state the main result of this section.
 \newtheorem{thm}{Theorem}[section]
\begin{thm}\label{jullemma5.7}(Discrete fractional Gr\"onwall inequality with constant delay term) Let $y^n,(i-1)N+1\leq n \leq iN,0\leq i\leq K$ be a non-negative real sequence and satisfy
\begin{equation}\label{eqs20240607a}
  D_{N}^{\alpha}y^n\leq \lambda y^{n-N}+\left|(R_t^\alpha)^n\right|+z^n,\ n=1,2,\cdots,KN,
\end{equation}
where $(R_t^\alpha)^n$ is the local truncation error of L1 scheme,  $\lambda$ is a non-negative constant and $z^n$ is a non-negative real bounded sequence. If $0<\alpha\le \frac{1}{2}$, then
\begin{equation}\label{eqs5 35}
    y^n \lesssim \left\{ \begin{aligned}
	&\Phi_n C_{\alpha,n,\tau,\lambda}+K_{\beta _1,n}\rho t_n^{\alpha-1},&i=1,\\
    &\Phi_n C_{\alpha,n,\tau,\lambda}+K_{\beta _1,n}\rho+K_{\beta _2,n}\rho^{1+\alpha}t_{n-N}^{\alpha-1}+K_{\beta _1,n}\rho t_{n-N}^{2\alpha-1},&i=2,\\
&\Phi_n C_{\alpha,n,\tau,\lambda}+K_{\beta _1,n}\rho+K_{\beta _2,n}\rho^{1+\alpha}
+\rho^{1+\alpha}t_{n-(i-1)N}^{(i-1)\alpha-1} +K_{\beta _2,n}\rho^{1+\alpha}t_{n-(i-1)N}^{(i-1)\alpha-1}+K_{\beta _1,n}\rho t_{n-(i-1)N}^{i\alpha-1},&i\geq 3.
\end{aligned} \right.
\end{equation}
If $\frac{1}{2}<\alpha<1$, then
\begin{equation}\label{eqs5 36}
y^n \lesssim \left\{ \begin{aligned}
	&\Phi_nC_{\alpha,n,\tau,\lambda}+K_{\beta _3,n}\rho t_n^{\alpha-1} ,&i=1,\\
    &\Phi_nC_{\alpha,n,\tau,\lambda}+K_{\beta _3,n}\rho+\rho^{2-\alpha} t_{n-N}^{3\alpha-2}, &i=2,\\
	&\Phi_nC_{\alpha,n,\tau,\lambda}+K_{\beta _3,n}\rho +\rho^{2-\alpha},&i\geq 3,
\end{aligned} \right.
\end{equation}
where $C_{\alpha,n,\tau,\lambda}=1+\frac{\lambda}{\Gamma(\alpha+1)}t_{n-N}^{\alpha}+\cdots+\frac{\lambda^{i-1}(\Gamma(2-\alpha))^{i-2}}{\Gamma((i-1)\alpha+1)}t_{n-(i-1)N}^{(i-1)\alpha}$ and $\Phi_n:=y^0+\lambda \underset{{1\leq j \leq N}}{\max}y^{n-N} \frac{t_N^\alpha}{\Gamma(1+\alpha)}+\underset{1\leq j\leq n}{\max}z^j\frac{t_n^\alpha}{\Gamma(1+\alpha)}$.
\end{thm}
\begin{proof}
It follows from the definition of L1 approximation (\ref{eqs3 1}) that (\ref{eqs20240607a}) can be rewritten as
\begin{equation}\label{eqs5 31}
    \sum_{k=1}^j{a_{j-k}}\nabla _ty_k\leq \lambda y^{j-N}+\left|(R_t^\alpha)^j\right|+z^j.
\end{equation}
Multiplying (\ref{eqs5 31}) by $P_{n-j} $ and summing $j$ from $1$ to $n$, we get
\begin{equation}\label{eqs5 32}
  \sum_{j=1}^n{P_{n-j}\sum_{k=1}^j{a_{j-k}}\nabla _ty_k}\leq \lambda \sum_{j=1}^n{P_{n-j}}y^{j-N}+\sum_{j=1}^n{P_{n-j}}\left|(R_t^\alpha)^j\right|+\sum_{j=1}^n P_{n-j}z^j.
\end{equation}
Applying Lemma \ref{jullemma5.3} gives
\begin{equation}\label{eqs5 33}
\begin{aligned}
y^n&\leq y_0+\lambda \sum_{j=1}^n{P_{n-j}}y^{j-N}+\sum_{j=1}^n{P_{n-j}}\left|(R_t^\alpha)^j\right|+\sum_{j=1}^n P_{n-j}z^j\\
   &\leq y_0+\lambda \sum_{j=N+1}^n{P_{n-j}}y^{j-N}+\lambda \sum_{j=1}^N{P_{n-j}}\max_{1\leq j \leq N}y^{j-N}+\sum_{j=1}^n{P_{n-j}}\left|(R_t^\alpha)^j\right|+\sum_{j=1}^n P_{n-j}\max_{1\leq j\leq n}z^j\\
   &\leq y_0+\lambda \sum_{j=N+1}^n{P_{n-j}}y^{j-N}+\lambda \max_{1\leq j \leq N}y^{j-N} \frac{t_N^\alpha}{\Gamma(1+\alpha)} +\sum_{j=1}^n{P_{n-j}}\left|(R_t^\alpha)^j\right|+\max_{1\leq j\leq n}z^j\frac{t_n^\alpha}{\Gamma(1+\alpha)}\\
   &= \Phi_n+\lambda \sum_{j=N+1}^n{P_{n-j}}y^{j-N}+\sum_{j=1}^n{P_{n-j}}\left|(R_t^\alpha)^j\right|,
\end{aligned}
 \end{equation}
 where
$\Phi_n:=y^0+\lambda \underset{{1\leq j \leq N}}{\max}y^{j-N} \frac{t_N^\alpha}{\Gamma(1+\alpha)} +\underset{1\leq j\leq n}{\max}z^j\frac{t_n^\alpha}{\Gamma(1+\alpha)}$.

Next, for the fixed $n$, let $Y=\left( y^n,y^{n-1},\cdots ,y^1 \right)^T $, $\Phi_n Z_2 = \Phi_n \left( 1,1,\cdots ,1 \right)^T $ and note that
$$
Z_1=\left( \sum_{j=1}^n{P_{n-j}}\left|(R_t^\alpha)^j\right|,\sum_{j=1}^{n-1}{P_{n-1-j}}\left|(R_t^\alpha)^j\right|,\cdots ,P_0\left|(R_t^\alpha)^1\right| \right)^T.
$$
Then (\ref{eqs5 33}) can be rewritten as the following matrix form
\begin{equation}\label{eqs20240607b}
  Y\leq \lambda JY+Z_1+\Phi_n Z_2.
\end{equation}
By using (\ref{eqs20240607b}) repeatedly, one has
$$
\begin{aligned}
Y&\leq \lambda JY+Z_1+\Phi_n Z_2\\
&\leq \lambda J\left( \lambda JY+Z_1+ \Phi_nZ_2 \right) +Z_1+ \Phi_nZ_2\\
&=\lambda^2 J^2Y+\sum_{q=0}^1{\lambda^q J^q\left( Z_1+ \Phi_nZ_2 \right)}\\
&\leq \cdots \leq \sum_{q=0}^{i-1}{\lambda^q J^q\left( Z_1+ \Phi_nZ_2 \right)}.
\end{aligned}
$$
Since the first elements of $Y,\ J^q Z_1,\ R_{n-qN}^{(q)}$ are $y^n$, $R_{n-qN}^{(q)}$ and $\varepsilon_{n-qN}^{(q)}$, respectively, we have
\begin{equation}\label{eqs20240607c}
\begin{aligned}
  y^n\leq &\sum_{q=0}^{i-1}{\lambda^q \left(R_{n-qN}^{(q)} + \Phi_n\varepsilon_{n-qN}^{(q)} \right)}\\
  \leq & \underset{0\leq q\leq i-1}{\max}\lambda^q\sum_{q=0}^{i-1}{R_{n-qN}^{(q)}}+\Phi_n\sum_{q=0}^{i-1}{\lambda^q\varepsilon_{n-qN}^{(q)}}.
  \end{aligned}
\end{equation}
The combination of Lemma \ref{jullemma20240605b}, Lemma \ref{jullemma20240605c} and (\ref{eqs20240607c}) indicates that (\ref{eqs5 35}) and (\ref{eqs5 36}) hold.
\end{proof}

\section{Convergence analysis}
In this section, we ignore the stability analysis and only discuss the convergence of the numerical scheme on uniform temporal mesh, because Ref. \cite{Tan2022L1} has shown that this scheme is stable on uniform and non-uniform temporal meshes. First we introduce an useful lemma.
\begin{lem}\label{jullemma4.1}\cite{Huang2022robust}
Suppose that $v^j\in L^2(\Omega),0\leq j\leq KN$. Then the L1 scheme satisfies
$$
{\left(D_{N}^{\alpha}v^{n},v^{n}\right)\geq\left(D_{N}^{\alpha}\|v^{n}\|_0\right)\|v^{n}\|_0},\quad 1\leq n\leq KN.
$$
\end{lem}

Now we investigate the convergence of the fully discrete scheme (\ref{eqs3 7}).  Define the Ritz projection operator $R_h:H_0^1(0,L)\to X_h$ by
\begin{equation}\label{eqs5 1}
B\left(R_h v, w \right)=B(v, w),\quad \forall w \in X_h, v \in H_0^1(0, L).
\end{equation}
In fact, $R_h$ has the following well-known property: if $u^n\in H_0^1(\Omega) \cap H^2(\Omega)$, then
\begin{equation}\label{eqs5 3}
    \|u^n-R_hu^n\|_0\leq Ch^2\|u^n\|_2.
\end{equation}
\begin{thm}\label{thm5.1} Assume that $u$ is the solution of (\ref{eqs1 1})--(\ref{eqs1 2}) with the regularity (\ref{eqs2 1}). Then the numerical solution $U_h^n$ of (\ref{eqs3 7}) satisfies that, for $ 0<\alpha\le \frac{1}{2}$
\begin{equation}\label{eqs20240608a}
  \| u^n-U_h^n \|_0 \lesssim \left\{ \begin{aligned}
	&h^2+h^2C_{\alpha,n,\tau,b}+K_{\beta _1,n}\rho t_n^{\alpha-1},\qquad\qquad\qquad\qquad\qquad\qquad\qquad\qquad\  i=1,\\
    &h^2+h^2C_{\alpha,n,\tau,b}+K_{\beta _1,n}\rho+K_{\beta _2,n}\rho^{1+\alpha}t_{n-N}^{\alpha-1}+K_{\beta _1,n}\rho t_{n-N}^{2\alpha-1},\qquad\qquad\quad i=2,\\
&\begin{aligned}
h^2+h^2C_{\alpha,n,\tau,b}&+K_{\beta _1,n}\rho+K_{\beta _2,n}\rho^{1+\alpha}+\rho^{1+\alpha}t_{n-(i-1)N}^{(i-1)\alpha-1}\\
&+K_{\beta _2,n}\rho^{1+\alpha}t_{n-(i-1)N}^{(i-1)\alpha-1}+K_{\beta _1,n}\rho t_{n-(i-1)N}^{i\alpha-1},\qquad\qquad\qquad i\geq 3,
\end{aligned}
\end{aligned} \right.
\end{equation}
and for $\frac{1}{2}<\alpha<1$
\begin{equation}\label{eqs20240608b}
  \| u^n-U_h^n \|_0 \lesssim \left\{ \begin{aligned}
	&h^2+h^2C_{\alpha,n,\tau,b}+K_{\beta _3,n}\rho t_n^{\alpha-1} ,&i=1,\\
    &h^2+h^2C_{\alpha,n,\tau,b}+K_{\beta _3,n}\rho+\rho^{2-\alpha} t_{n-N}^{3\alpha-2}, &i=2,\\
	&h^2+h^2C_{\alpha,n,\tau,b}+K_{\beta _3,n}\rho +\rho^{2-\alpha},&i\geq 3,
\end{aligned} \right.
\end{equation}
where
$C_{\alpha,n,\tau,b}=1+\frac{|b|}{\Gamma(\alpha+1)}t_{n-N}^{\alpha}+\cdots+\frac{|b|^{i-1}(\Gamma(2-\alpha))^{i-2}}{\Gamma((i-1)\alpha+1)}t_{n-(i-1)N}^{(i-1)\alpha}$.
\end{thm}
\begin{proof}
Let $r_h^n=u^n-R_{h}u^{n}$, $\epsilon_h^n=R_hu^n-U_h^n\in X_{h}$. Combining with (\ref{eqs1 1}), (\ref{eqs3 7}) and (\ref{eqs5 1}), the standard finite element analysis implies that
$$
(D_N^\alpha(u^n-U_h^n),v_h)+B(u^n-U_h^n,v_h)=b(u^{n-N}-U_h^{n-N},v)-((R_t^\alpha)^n,v_h),\ \forall v_h \in X_h,
$$
i.e.,
\begin{equation}\label{eqs20240608c}
  (D_N^\alpha \epsilon_h^n,v_h)+B(\epsilon_h^n),v_h)=b(\epsilon_h^{n-N},v_h)+(b r_h^{n-N}-D_N^\alpha r_h^n-(R_t^\alpha)^n,v_h),\ \forall v_h \in X_h.
\end{equation}
Take $v_h=\epsilon_h^n$ into (\ref{eqs20240608c}), according to (\ref{eq240606a}), the Lemma \ref{jullemma4.1} and the Cauchy-Schwarz inequality, it gives
$$
   D_{N}^{\alpha}\|\epsilon^{n}_h\|_0\le |b|\|\epsilon_h^{n-N}\|_0+\|(R_t^\alpha)^n\|_0+\left(\|br_h^{n-N}\|_0+\|D_N^\alpha r_h^n\|_0\right).
$$
By (\ref{eqs5 3}), we know that $\|r_h^{n-N}\|_0\leq Ch^2\|u^{n-N}\|_2$ and $D_N^\alpha \|r_h^{n}\|_0 \leq (_0^CD_t^\alpha-D_N^\alpha)\|r_h^{n}\|_0+_0^CD_t^\alpha\|r_h^{n}\|_0 \lesssim (R_t^\alpha)^{n}+h^2.$
In Theorem \ref{jullemma5.7}, taking $y^n=\|\epsilon_h^n\|_0$, $\lambda=|b|$ and $z^n=\|br_h^{n-N}\|_0+\|D_N^\alpha r_h^n\|_0$, then for $0<\alpha \leq \frac{1}{2}$
\begin{equation}\label{eqs20240608d}
  \| \epsilon_h^n \|_0 \lesssim \left\{ \begin{aligned}
	&h^2C_{\alpha,n,\tau,b}+K_{\beta _1,n}\rho t_n^{\alpha-1},&i=1,\\
    &h^2C_{\alpha,n,\tau,b}+K_{\beta _1,n}\rho+K_{\beta _2,n}\rho^{1+\alpha}t_{n-N}^{\alpha-1}+K_{\beta _1,n}\rho t_{n-N}^{2\alpha-1},&i=2,\\
    &h^2C_{\alpha,n,\tau,b}+K_{\beta _1,n}\rho+K_{\beta _2,n}\rho^{1+\alpha}+\rho^{1+\alpha}t_{n-(i-1)N}^{(i-1)\alpha-1} +K_{\beta _2,n}\rho^{1+\alpha}t_{n-(i-1)N}^{(i-1)\alpha-1}+K_{\beta _1,n}\rho t_{n-(i-1)N}^{i\alpha-1},&i\geq 3,
\end{aligned} \right.
\end{equation}
and for $\frac12 <\alpha<1$
\begin{equation}\label{eqs20240608j}
  \| \epsilon_h^n \|_0 \lesssim \left\{ \begin{aligned}
	&h^2C_{\alpha,n,\tau,b}+K_{\beta _3,n}\rho t_n^{\alpha-1} ,&i=1,\\
    &h^2C_{\alpha,n,\tau,b}+K_{\beta _3,n}\rho+\rho^{2-\alpha} t_{n-N}^{3\alpha-2}, &i=2,\\
	&h^2C_{\alpha,n,\tau,b}+K_{\beta _3,n}\rho +\rho^{2-\alpha},&i\geq 3,
\end{aligned} \right.
\end{equation}
Applying the fact that $\|u^n-U_h^n\|_0 \leq \|r_h^n\|_0+\|\epsilon_h^n\|_0$,
so (\ref{eqs20240608a}) and (\ref{eqs20240608b}) can be obtained immediately.
\end{proof}

\begin{remark}
  By observing the above convergence results, it shows that, for $(i-1)N+1<n\leq iN$, the temporal error is near $O(\rho^{i\alpha})$ as $n$ near $(i-1)N+1$, meanwhile the error is almost $O(\rho)$ when $n$ is away from $(i-1)N$. From Ref. \cite{Tan2022L1} and \cite{Bu2024Finite} (and also see (\ref{eqs2 1})), we know that the temporal smoothness of the solution to (\ref{eqs1 1})--(\ref{eqs1 2}) near $t_{(i-1)N}^{+}$ will be continuously improved as increasing $i$. Thus error results obtained above are consistent with the theory.
  Note the fact that $\ln n$ increases slowly, if we regard $\ln n$ as a constant, then Theorem \ref{thm5.1} implies that
  \begin{equation}\label{eqs20240608d}
    \max_{(i-1)N+1 \leq n\leq iN}\|u^n-U_h^n\|_0 \lesssim \rho^{\min\{1,i\alpha\}} + h^2, \quad i=1,2,\cdots,K.
  \end{equation}
It is obviously more accurate than to consider the global error $O(\rho^\alpha)$. Besides, due to the applying of the developed discrete fractional Gr\"onwall inequality, the present error bounds do not contain a fast increasing Mittag-Leffler function. Therefore, these characteristics indicate that our results are not trivial and cannot be considered as a special case of Refs. \cite{Tan2022L1,Bu2024Finite}.
\end{remark}


\section{Numerical experiment}
In this section, we will verify the established theoretical results by some numerically tests. Here, for the considered problem (\ref{eqs1 1})--(\ref{eqs1 2}), we take $b=-1$, $\tau=1$, choose the spatial domain $[0,1]$, the temporal interval $(0,3]$, the initial condition $u(x,t)=\sin(\pi x)$ and the right hand function
$$
 f(x,t) = \left\{ \begin{aligned}
	&(p\pi^2-a(x))\left[ 1-\omega_{\alpha+1}(t)\right]\sin\pi x,&0 < t\leq 1,\\
    &(p\pi^2-a(x))\left[1-\omega_{\alpha+1}(t)+\omega_{2\alpha+1}(t-1)\right]\sin\pi x,&1 < t\leq 2,\\
    &(p\pi^2-a(x))\left[1-\omega_{\alpha+1}(t)+\omega_{2\alpha+1}(t-1)-\omega_{3\alpha+1}(t-2)\right]\sin\pi x,&2 < t\leq 3.
\end{aligned} \right.
$$
It is easy to know that the exact solution of the above case is
$$
u(x,t) = \left\{ \begin{aligned}
	&\left[ 1-\omega_{\alpha+1}(t) \right]\sin\pi x,&0 < t\leq 1,\\
    &\left[1-\omega_{\alpha+1}(t)+\omega_{2\alpha+1}(t-1)\right]\sin\pi x,&1 < t\leq 2,\\
    &\left[1-\omega_{\alpha+1}(t)+\omega_{2\alpha+1}(t-1)-\omega_{3\alpha+1}(t-2)\right]\sin\pi x,&2 < t\leq 3.
\end{aligned} \right.
$$
Before using the numerical scheme (\ref{eqs3 7}) to calculate the above problem, we give some marks firstly. Define
$$
E(h,\rho,i)=\max_{(i-1)N+1 \leq n\leq iN}\|u^n-U_h^n\|_0,
$$
$$
rate_t=\log_2\left(\frac{E(h,\rho,i)}{E(h,\rho/2,i)}\right),\  rate_s= \log_2\left(\frac{E(h,\rho,i)}{E(h/2,\rho,i)}\right),
$$
where $rate_t$ will be used to examine the local temporal convergence on interval $(i-1,i], i=1,2,3$ and $rate_s$ will be considered to investigate the convergence in spatial case.

Now, based on $p=1/100,a(x)=-x/100$ for $\alpha=0.4,0.5,0.6$ and $p=1/500,a(x)=-x/500$ for $\alpha=0.8$, we observe the numerical results in Tables \ref{t1}--\ref{t6}. First we choose sufficiently small $h=1/1000$ to ensure that the spatial errors can be ignored.
The numerical results in Tables \ref{t1}--\ref{t4} with different $\alpha$ show that the convergence rates in time is $\min\{1,i\alpha\}$ for the calculated interval $(i-1,i]$, and the temporal convergence at $t_{iN},i=1,2,3$ given in Table 5 is almost 1. It is clear that these results are consistent with (\ref{eqs20240608d}), i.e., it is agree with Theorem \ref{thm5.1}. Next we examine the spatial convergence. In order to make the spatial error dominant, the time step size is taken as $1/30000$. Table \ref{t6} presents the spatial convergence rates based on the calculated temporal interval $(2,3]$ and different $\alpha$. These results is obviously consistent with the expected cases.



\begin{table}
\caption{The temporal accuracy and error of the numerical solution for $\alpha=0.4$.}\label{t1}
\begin{center}
\normalsize{
\begin{tabular}{l|llllllll}
\hline
$N$ &$i=1$         &     &$i=2$           &    &$i=3$            &  \\
\ \            &$E(h,\rho,i)$  &$rate_t$       &$E(h,\rho,i)$   &$rate_t$      &$E(h,\rho,i)$    &$rate_t$  \\
\hline
$100$   &2.5794e-02    &--      &3.7640e-03   &--       &1.4451e-03   &--\\

$200$   &1.9617e-02    &0.395  &2.2871e-03   &0.719   &7.3976e-04   &0.966\\

$400$   &1.4907e-02    &0.396  &1.3751e-03   &0.734   &3.7550e-04   &0.978\\

$800$   &1.1320e-02    &0.397  &8.1988e-04   &0.746   &1.8958e-04   &0.986\\

$1600$  &8.5923e-03    &0.398  &4.8559e-04   &0.756   &9.5385e-05   &0.991\\

$O\left(\rho^{\min\{1,i\alpha\}}\right)$  &--    &0.400  &--  &0.800  &--   &1.000\\
\hline
\end{tabular}}
\end{center}
\end{table}

\begin{table}
\caption{The temporal accuracy and error of the numerical solution for $\alpha=0.5$.}\label{t2}
\begin{center}
\normalsize{
\begin{tabular}{l|llllllll}
\hline
$N$ &$i=1$         &     &$i=2$           &    &$i=3$            &  \\
\ \            &$E(h,\rho,i)$  &$rate_t$       &$E(h,\rho,i)$   &$rate_t$      &$E(h,\rho,i)$    &$rate_t$  \\
\hline
$100$   &1.6967e-02    &--      &1.5736e-03   &--      &1.5378e-03   &--\\

$200$   &1.2029e-02    &0.496  &8.0744e-04   &0.963   &7.9431e-04   &0.953\\

$400$   &8.5222-03     &0.497  &4.1099e-04   &0.974   &4.0624e-04   &0.967\\

$800$   &6.0342e-03    &0.498  &2.0807e-04   &0.982   &2.0636e-04   &0.977\\

$1600$  &4.2709e-03    &0.499  &1.0494e-04   &0.988   &1.0433e-04   &0.984\\

$O\left(\rho^{\min\{1,i\alpha\}}\right)$  &--    &0.500  &--  &1.000  &--   &1.000\\
\hline
\end{tabular}}
\end{center}
\end{table}

\begin{table}
\caption{The temporal accuracy and error of the numerical solution for $\alpha=0.6$.}\label{t3}
\begin{center}
\normalsize{
\begin{tabular}{l|llllllll}
\hline
$N$ &$i=1$         &    &$i=2$           &    &$i=3$            &  \\
\ \            &$E(h,\rho,i)$  &$rate_t$       &$E(h,\rho,i)$   &$rate_t$      &$E(h,\rho,i)$    &$rate_t$  \\
\hline
$100$   &1.0287e-02    &--      &1.6308e-03   &--       &1.3543e-03   &--\\

$200$   &6.8003e-03    &0.597  &8.4307e-04   &0.952   &7.2139e-04   &0.909\\

$400$   &4.4923e-03    &0.598  &4.3309e-04   &0.961   &3.7811e-04   &0.932\\

$800$   &2.9664e-03    &0.599  &2.2142e-04   &0.968   &1.9584e-04   &0.949\\

$1600$  &1.9582e-03    &0.599  &1.1279e-04   &0.973   &1.0054e-04   &0.962\\

$O\left(\rho^{\min\{1,i\alpha\}}\right)$  &--    &0.600  &--  &1.000  &--   &1.000\\
\hline
\end{tabular}}
\end{center}
\end{table}

\begin{table}
\caption{The temporal accuracy and error of the numerical solution for $\alpha=0.8$.}\label{t4}
\begin{center}
\normalsize{
\begin{tabular}{l|llllllll}
\hline
$N$ &$i=1$         &     &$i=2$           &    &$i=3$            &  \\
\ \            &$E(h,\rho,i)$  &$rate_t$       &$E(h,\rho,i)$   &$rate_t$      &$E(h,\rho,i)$    &$rate_t$  \\
\hline
$100$   &3.0767e-03    &--      &2.3231e-03   &--       &8.6031e-04   &--\\

$200$   &1.7680e-03    &0.799  &1.1965e-03   &0.957   &4.2039e-04   &1.033\\

$400$   &1.0156e-03    &0.800  &6.1440e-04   &0.962   &2.0508e-04   &1.036\\

$800$   &5.8319e-04    &0.800  &3.1458e0-4   &0.966   &1.0188e-04   &1.009\\

$1600$  &3.3475e-04    &0.801  &1.6063e-04   &0.970  &5.2873e-05   &0.946\\

$O\left(\rho^{\min\{1,i\alpha\}}\right)$  &--    &0.800  &--  &1.000  &--   &1.000\\
\hline
\end{tabular}}
\end{center}
\end{table}

\begin{table}
\caption{ The temporal accuracy and error of the numerical solution for $\alpha=0.5$ at $t_{iN}$.}\label{t5}
\begin{center}
\normalsize{
\begin{tabular}{l|llllllll}
\hline
$N$ &$i=1$         &     &$i=2$           &    &$i=3$            &  \\
\ \            &$\|u^N-U_h^N\|_0$  &        &$\|u^{2N}-U_h^{2N}\|_0$   &       &$\|u^{3N}-U_h^{3N}\|_0$    &   \\
\hline
$100$   &1.6837e-03    &--      &1.5738e-03   &--       &1.0376e-03   &--\\

$200$   &8.3949e-04    &1.004  &8.0766e-04   &0.962   &5.6073e-04   &0.888\\

$400$   &4.1894e-04    &1.003  &4.1122e-04   &0.974   &2.9527e-04   &0.925\\

$800$   &2.0921e-04    &1.002  &2.0829e-04   &0.981   &1.5288e-04   &0.950\\

$1600$  &1.0454e-04    &1.001  &1.0517e-04   &0.986   &7.8251e-05   &0.966\\

$O\left(\rho\right)$  &--    &1.000  &--  &1.000  &--   &1.000\\
\hline
\end{tabular}}
\end{center}
\end{table}

\begin{table}
\caption{The spatial accuracy and error of the numerical solution for $N=30000$.}\label{t6}
\begin{center}
\normalsize{
\begin{tabular}{l|llllllll}
\hline
$M$ &$\alpha=0.4$      &     &$\alpha=0.5$   &    &$\alpha=0.6$   &       &$\alpha=0.8$  &  \\
\ \                 &$E(h,\rho,i)$           &$rate_s$       &$E(h,\rho,i)$       &$rate_s$      &$E(h,\rho,i)$        &$rate_s$         &$E(h,\rho,i)$       &$rate_s$\\
\hline
$8$   &5.8354e-03     &--      &4.6009e-03   &--       &3.6821e-03    &--        &1.9808e-03    &--   \\

$16$   &1.4620e-03    &1.997  &1.1527e-03   &1.997   &9.2256e-04   &1.997     &4.9811e-04    &1.992\\

$32$   &3.6805e-04    &1.990  &2.8804e-04   &2.001   &2.3054e-04   &2.001     &1.2626e-04    &1.980\\

$64$   &9.5144e-05    &1.952  &7.1871e-05   &2.003   &5.7530e-05   &2.003     &3.3239e-05    &1.925\\

$O\left(h^2 \right)$  &--    &2.000  &--  &2.000  &--   &2.000   &--   &2.000 \\
\hline
\end{tabular}}
\end{center}
\end{table}

\section{Conclusion}
In this paper, we use the L1/finite element scheme to solve the constant delay reaction-subdiffusion equation (\ref{eqs1 1})--(\ref{eqs1 2}), and investigate the pointwise-in-time and piecewise-in-time convergence of this numerical scheme under the uniform temporal mesh and the non-uniform multi-singularity (\ref{eqs2 1}) on the exact solution. In order to discuss the error estimate, the local truncation error of L1 scheme on uniform mesh are derived based on the regularity assumption (\ref{eqs2 1}) of the solution, and a novel discrete fractional Gr\"onwall inequality with constant delay term is established. It is worth to note that the developed Gr\"onwall inequality does not contain the increasing Mittag-Leffler function. Thus, by using this Gr\"onwall inequality, the obtained error estimates also have no Mittag-Leffler function. Another important role of this inequality is that it can ensure that the obtained error estimates have the characteristics of piecewise-in-time, i.e., the piecewise-in-time convergence order is $\min\{1,i\alpha\}$ for the considered interval $((i-1)\tau, i\tau]$, which is obviously consistent with the fact, i.e., the smoothness of the solution will be improved as the increasing $i$. Besides, we provide some numerical experiments to confirm these results.

\section{Acknowledgments}
This research is supported by the Research Foundation of Education Commission of Hunan Province of China (Grant No. 23A0126), and the National Natural Science Foundation of China (Grant Nos. 11971412, 12171466).

\end{document}